\documentclass[12pt]{article}

\usepackage{amsmath,amssymb,amsthm,mathtools}
\usepackage{enumitem}
\usepackage{hyperref}

\usepackage{setspace}
\newtheorem{theorem}{Theorem}[section]
\newtheorem{lemma}[theorem]{Lemma}
\newtheorem{proposition}[theorem]{Proposition}
\newtheorem{corollary}[theorem]{Corollary}

\theoremstyle{definition}
\newtheorem{definition}[theorem]{Definition}
\newtheorem{remark}[theorem]{Remark}

\newcommand{\Ften} {\fontsize{10}{11}\selectfont  }

\usepackage{fancyhdr}
\title{Markov Properties of $k$-Record Processes via Order Statistics}
\author{R. Labouriau\footnote{Department of Mathematics, Aarhus University.\\
                                   e-mail:   \texttt{rodrigo.labouriau@math.au.dk} and
                                                 \texttt{rodrigo.labouriau@rlstatlab.com}}
}
\date{Spring 2026}

\begin{document}

\maketitle

\begin{abstract}  \Ften
The theory of $k$-record values (Type~2 $k$-records) plays an important
role in the study of partial extremes and in statistical inference based
on record data. A common approach in the literature reduces the analysis
of $k$-records associated with a distribution function $F$ to that of
ordinary record values from the transformed distribution
$F_{1:k}(x)=1-(1-F(x))^k$. This representation is widely used in
deriving distributional results and inferential procedures, often without
an explicit construction of the underlying stochastic mechanism, and
relies on a structural property of order statistics that, although
classical, is typically invoked without proof in the record literature.

In this paper we provide a direct derivation of the probabilistic
structure of $k$-record processes based on the sequence of running order
statistics $U_n=X_{n-k+1:n}$, the $k$-th largest among the first $n$
observations. We show that this process forms a Markov chain with an
explicit transition kernel, making explicit the conditional structure of
upper order statistics underlying the classical reduction. Under the
continuity assumption on $F$, the usual Type~2 $k$-record times coincide
almost surely with the record times of $(U_n)_{n\ge k}$.

This yields a transparent construction of the $k$-record process as the
record process of a Markov chain. Classical distributional results,
including the representation through $F_{1:k}$ and the joint density of
the first $m$ $k$-record values, are recovered in a unified framework.
\end{abstract}

 \newpage
 \tableofcontents
 \newpage
\section{Introduction}\label{sec:introduction}

Record values constitute a classical topic in probability theory, going back to early work of Chandler, 1952 \cite{Chandler1952} and Rényi 1962 \cite{Renyi1962}; see also Nevzorov, 2001 \cite {Nevzorov2001} for a comprehensive account. The subject has a wide range of applications in statistics, reliability theory and extreme value analysis. A standard book-length reference is the monograph of Arnold, Balakrishnan and Nagaraja \cite{ArnoldBalakrishnanNagaraja1998}. More recently, increasing attention has been devoted to generalisations of ordinary records, among which the so-called $k$-record values, or Type~2 $k$-records, play a prominent role; see, for example, Hofmann and Balakrishnan, 2004 \cite{HofmannBalakrishnan2004} and Ahmadi and Doostparast, 2008 \cite{AhmadiDoostparast2008}.

Given a sequence of independent and identically distributed random variables, $k$-record values arise naturally when one considers the evolution of the $k$-th largest observation among the first $n$ observations. These generalised records retain a number of the attractive features of ordinary records, whilst providing a richer structure that is better suited for statistical inference based on partial extreme information. In particular, $k$-records have been used in inferential problems involving likelihood methods, Bayesian procedures and information measures; see \cite{HofmannBalakrishnan2004,AhmadiDoostparast2008,WangYe2015}.

A common approach in the literature consists in reducing the study of $k$-record values associated with a distribution function $F$ to the study of ordinary record values from the transformed distribution
\[
F_{1:k}(x)=1-(1-F(x))^k.
\]
This representation is widely used as a starting point for the derivation of joint distributions and likelihood functions. For example, Arnold et al.\ \cite[p.~43]{ArnoldBalakrishnanNagaraja1998} introduce Type~2 $k$-records in this framework, while later papers such as Hofmann and Balakrishnan \cite{HofmannBalakrishnan2004} and Ahmadi and Doostparast \cite{AhmadiDoostparast2008} use the corresponding distributional formulas in subsequent developments. Related ordinary-record arguments are also invoked in inferential work such as Wang and Ye \cite{WangYe2015}.

While this reduction is standard, it implicitly relies on a structural property of order statistics that is often used without explicit justification. Namely, one needs that, conditionally on the $k$-th largest observation, the remaining upper order statistics are distributed as order statistics from the truncated distribution above that value. Although this fact is classical in the theory of order statistics (see, e.g., Arnold et al.~\cite[pp.~18--19]{ArnoldBalakrishnanNagaraja1998} or David and Nagaraja~\cite[Sec.~2.4--2.5]{DavidNagaraja2003}), to the best of our knowledge it is not explicitly formulated and proved in the literature on $k$-records, where it is typically invoked implicitly.
One of the contributions of the present paper is to isolate this property in a precise form and to provide a direct argument adapted to the record-setting framework.

The purpose of the present paper is to provide a direct and rigorous derivation of the probabilistic structure underlying the $k$-record process. Our approach is based on the analysis of the sequence of running order statistics
\[
U_n=X_{n-k+1:n}, \qquad n\ge k,
\]
that is, the $k$-th largest observation among the first $n$ observations. We show in Section~\ref{sec:markov-order-statistics} that this process is a Markov chain with an explicit transition mechanism. We then compare, in Section~\ref{sec:record-times}, two natural constructions of record times: the usual Type~2 $k$-record times and the ordinary record times associated with the process $(U_n)_{n\ge k}$. Under the continuity assumption on $F$, we prove that these two constructions coincide almost surely.

This point of view yields a transparent construction of the $k$-record process and provides a natural framework in which classical results may be recovered as consequences. In particular, in Section~\ref{sec:k-record-process} we identify the $k$-record values with the record values of the Markov chain $(U_n)$, and in Section~\ref{sec:classical-results} we recover the standard representation through $F_{1:k}$ together with the usual joint density formulae for the first $m$ $k$-records.

Although the introductory material in this section is necessarily brief, our aim is not merely to rederive known formulas. Rather, the objective is to make explicit the underlying stochastic mechanism that links running order statistics, record times and the usual $k$-record distributions.

\section{Preliminaries}\label{sec:preliminaries}

Let $\{X_n\}_{n\ge1}$ be a sequence of independent and identically distributed random variables with common distribution function $F$. Unless otherwise stated, we assume throughout that $F$ is continuous.

For each $n\ge1$, let
\[
X_{1:n}\le X_{2:n}\le \cdots \le X_{n:n}
\]
denote the order statistics associated with $X_1,\dots,X_n$.

Fix an integer $k\ge1$. A central role in what follows is played by the process
\[
U_n=X_{n-k+1:n}, \qquad n\ge k,
\]
that is, the $k$-th largest observation among the first $n$ observations.

We shall consider two sequences of record times associated with this construction.

\begin{definition}\label{def:type2-k-record-times}
Set $v_1^{(k)}=k$, and for $n\ge1$ define recursively
\[
v_{n+1}^{(k)}
=
\inf\Bigl\{j>v_n^{(k)}:X_j>X_{\,v_n^{(k)}-k+1:v_n^{(k)}}\Bigr\}.
\]
The corresponding record values are
\[
R_n^{(k)}=X_{\,v_n^{(k)}-k+1:v_n^{(k)}}, \qquad n\ge1.
\]
These are the usual Type~2 $k$-record times and values; see \cite{ArnoldBalakrishnanNagaraja1998,AhmadiDoostparast2008}.
\end{definition}

\begin{definition}\label{def:record-times-Un}
Set $\mu_1^{(k)}=k$, and for $n\ge1$ define
\[
\mu_{n+1}^{(k)}
=
\inf\{j>\mu_n^{(k)}:U_j>U_{\mu_n^{(k)}}\}.
\]
The associated record values are
\[
U_{\mu_n^{(k)}}, \qquad n\ge1.
\]
\end{definition}

The sequences $\{v_n^{(k)}\}_{n\ge1}$ and $\{\mu_n^{(k)}\}_{n\ge1}$ represent two natural constructions of record times. The first is defined directly from the original observations, in the usual manner of Type~2 $k$-records, whereas the second is defined as the ordinary record-time sequence of the process $(U_n)_{n\ge k}$.

The connection between these two constructions is one of the main themes of the paper. In Section~\ref{sec:record-times} we shall prove that, under continuity of $F$, they coincide almost surely. This will allow us to interpret the $k$-record process as the record process associated with the Markov chain $(U_n)$ studied in Section~\ref{sec:markov-order-statistics}.

\section{The order-statistic process and its Markov property}
\label{sec:markov-order-statistics}

In this section we study the stochastic evolution of the upper order statistics generated by the sequence $(X_n)_{n\ge1}$. We show that the process formed by the largest $k$ order statistics is a time-homogeneous Markov chain. This result will serve as the basis for the analysis of $k$-record times and values in subsequent sections.

\subsection{The vector process of upper order statistics}
\label{subsec:vector-process}

Fix an integer $k\ge1$. For each $n\ge k$, define the random vector
\[
Y_n = \bigl(X_{n-k+1:n},\dots,X_{n:n}\bigr),
\]
taking values in the set
\[
\mathcal{E}_k
=
\{(y_1,\dots,y_k)\in\mathbb{R}^k : y_1 \le \cdots \le y_k\},
\]
endowed with the Borel $\sigma$-algebra inherited from $\mathbb{R}^k$.

Let $(\mathcal{F}_n)_{n\ge1}$ denote the natural filtration,
\[
\mathcal{F}_n = \sigma(X_1,\dots,X_n).
\]
Then $Y_n$ is $\mathcal{F}_n$-measurable for every $n\ge k$.

\medskip

We first formalise the update mechanism of the process. For $y=(y_1,\dots,y_k)\in\mathcal{E}_k$ and $x\in\mathbb{R}$, define
\[
\psi_k(y,x)
\]
as the vector obtained by taking the $k$ largest elements of the multiset $\{y_1,\dots,y_k,x\}$ and arranging them in increasing order. Equivalently, let $(z_1,\dots,z_{k+1})$ be the non-decreasing rearrangement of $(y_1,\dots,y_k,x)$, and set
\[
\psi_k(y,x) = (z_2,\dots,z_{k+1}).
\]

\medskip

\begin{lemma}
\label{lem:psi-measurable}
The mapping $\psi_k : \mathcal{E}_k \times \mathbb{R} \to \mathcal{E}_k$ is Borel measurable.
\end{lemma}

\begin{proof}
The mapping that sends a vector in $\mathbb{R}^{k+1}$ to its ordered version is measurable, as it can be expressed in terms of coordinate-wise minima and maxima. The operation of removing the smallest coordinate is continuous on the ordered region. Since $\psi_k$ is obtained by composing these measurable mappings, it is Borel measurable.
\end{proof}

\medskip

The following identity holds almost surely for all $n\ge k$:
\begin{equation}
\label{eq:Yn-recursion}
Y_{n+1} = \psi_k(Y_n, X_{n+1}).
\end{equation}

\begin{proposition}
\label{prop:vector-process-markov}
The process $(Y_n)_{n\ge k}$ is a time-homogeneous Markov chain with respect to the filtration $(\mathcal{F}_n)_{n\ge1}$. More precisely, for every bounded Borel function $g:\mathcal{E}_k\to\mathbb{R}$ and every $n\ge k$,
\[
\mathbb{E}\bigl[g(Y_{n+1}) \mid \mathcal{F}_n\bigr]
=
\int_{\mathbb{R}} g\bigl(\psi_k(Y_n,x)\bigr)\, dF(x)
\quad \text{a.s.}
\]
In particular,
\[
\mathbb{P}(Y_{n+1}\in A \mid \mathcal{F}_n)
=
\int_{\mathbb{R}} \mathbf{1}_A\bigl(\psi_k(Y_n,x)\bigr)\, dF(x),
\quad A\in\mathcal{B}(\mathcal{E}_k).
\]
\end{proposition}

\begin{proof}
Let $g$ be a bounded Borel function on $\mathcal{E}_k$. Using \eqref{eq:Yn-recursion}, we have
\[
g(Y_{n+1}) = g\bigl(\psi_k(Y_n,X_{n+1})\bigr).
\]
Since $Y_n$ is $\mathcal{F}_n$-measurable and $X_{n+1}$ is independent of $\mathcal{F}_n$ with distribution $F$, it follows from the standard properties of conditional expectation that
\[
\mathbb{E}\bigl[g(Y_{n+1}) \mid \mathcal{F}_n\bigr]
=
\int_{\mathbb{R}} g\bigl(\psi_k(Y_n,x)\bigr)\, dF(x)
\quad \text{a.s.}
\]
This establishes the Markov property, as the conditional distribution of $Y_{n+1}$ given $\mathcal{F}_n$ depends on the past only through $Y_n$. The expression for conditional probabilities follows by taking $g=\mathbf{1}_A$.
\end{proof}
The following proposition reformulates the Markov property of $(Y_n)$
by making the deterministic update mechanism explicit through a
mapping~$\Phi$. While equivalent in content to
Proposition~\ref{prop:vector-process-markov}, this formulation is more directly
suited for the strong Markov property argument used in
Section~\ref{sec:k-record-process}.

\begin{proposition}
\label{prop:vector-process-markovy}
Assume that $F$ is continuous. Then the process $(Y_n)_{n\ge k}$ is a time-homogeneous Markov chain with respect to the filtration $(\mathcal F_n)_{n\ge k}$, where
\[
\mathcal F_n=\sigma(X_1,\dots,X_n).
\]
\end{proposition}
\begin{proof}
For each $n\ge k$, the vector $Y_n=(Y_{n,1},\dots,Y_{n,k})$ consists of the $k$ largest observations among $X_1,\dots,X_n$, arranged in increasing order. Hence there exists a measurable mapping
\[
\Phi:\mathbb R^k\times\mathbb R\to\mathbb R^k
\]
such that
\[
Y_{n+1}=\Phi(Y_n,X_{n+1}), \qquad n\ge k.
\]
Indeed, $\Phi(y,x)$ is obtained by inserting $x$ into the ordered vector $y=(y_1,\dots,y_k)$ and then retaining the $k$ largest entries, again in increasing order.

Now let $B\in\mathcal B(\mathbb R^k)$. Since $X_{n+1}$ is independent of $\mathcal F_n$ and has distribution function $F$, we obtain
\[
\mathbb P(Y_{n+1}\in B\mid \mathcal F_n)
=
\mathbb P(\Phi(Y_n,X_{n+1})\in B\mid \mathcal F_n)
=
\int_{\mathbb R}\mathbf 1_{\{\Phi(Y_n,x)\in B\}}\,dF(x).
\]
The right-hand side is a measurable function of $Y_n$ only. Therefore,
\[
\mathbb P(Y_{n+1}\in B\mid \mathcal F_n)
=
P(Y_n,B)
\]
for a suitable transition kernel \(P\) on \(\mathbb R^k\). This proves that $(Y_n)_{n\ge k}$ is a time-homogeneous Markov chain with respect to $(\mathcal F_n)_{n\ge k}$.
\end{proof}

\subsection{The one-dimensional process $U_n$}
\label{subsec:one-dimensional-process}

We now turn to the process
\[
U_n=X_{n-k+1:n}, \qquad n\ge k,
\]
that is, the running $k$th largest observation.

The Markov property of $(U_n)_{n\ge k}$ is less immediate than that of the vector process $(Y_n)_{n\ge k}$, since in general a measurable function of a Markov chain need not be Markov. We therefore derive the transition mechanism of $(U_n)$ directly.

\medskip

Let
\[
Z_i=F(X_i), \qquad i\ge1.
\]
Since $F$ is continuous, the random variables $(Z_i)_{i\ge1}$ are independent and uniformly distributed on $(0,1)$. For $n\ge k$, define
\[
V_n=Z_{n-k+1:n},
\]
the running $k$th largest order statistic of the transformed sample. Then
\[
V_n=F(U_n) \qquad \text{a.s.}
\]

 The next lemma describes the conditional law of the upper order statistics above $V_n$, a classical property of order statistics; see, e.g., Arnold et al.~\cite{ArnoldBalakrishnanNagaraja1998}, pp.~18--19 (see also David and Nagaraja~\cite{DavidNagaraja2003}, Sec.~2.4--2.5, pp~40-47). This result is widely used in the literature on records, to the best of our knowledge without an explicit proof.
 \begin{lemma}
\label{lem:conditional-upper-uniform-order-stats}
Fix $n\ge k$. For $0<u<1$, the conditional distribution of
\[
\bigl(Z_{n-k+2:n},\dots,Z_{n:n}\bigr)
\quad \text{given } V_n=u
\]
coincides with the distribution of the order statistics of $k-1$ independent random variables uniformly distributed on $(u,1)$.
\end{lemma}

\begin{proof}
The joint density of the upper $k$ order statistics
\[
\bigl(Z_{n-k+1:n},Z_{n-k+2:n},\dots,Z_{n:n}\bigr)
\]
is
\[
f_n(u_1,\dots,u_k)
=
\frac{n!}{(n-k)!}\,u_1^{\,n-k}
\mathbf{1}_{\{0<u_1<\cdots<u_k<1\}}.
\]
Indeed, this is the usual joint density of uniform order statistics specialised to the last $k$ coordinates.

Integrating out $u_2,\dots,u_k$, we obtain the marginal density of $V_n=Z_{n-k+1:n}$:
\[
f_{V_n}(u)
=
\frac{n!}{(n-k)!(k-1)!}\,
u^{\,n-k}(1-u)^{k-1},
\qquad 0<u<1.
\]
Hence the conditional density of
\[
\bigl(Z_{n-k+2:n},\dots,Z_{n:n}\bigr)
\quad \text{given } V_n=u
\]
is
\[
\frac{f_n(u,u_2,\dots,u_k)}{f_{V_n}(u)}
=
\frac{(k-1)!}{(1-u)^{k-1}}
\mathbf{1}_{\{u<u_2<\cdots<u_k<1\}}.
\]
This is precisely the joint density of the order statistics of $k-1$ independent random variables having the uniform distribution on $(u,1)$.
\end{proof}

\medskip

We now translate the preceding statement back to the original scale.

\begin{lemma}
\label{lem:conditional-upper-order-stats}
Fix $n\ge k$, and let
\[
F_x(y)=\frac{F(y)-F(x)}{1-F(x)}, \qquad y\ge x,
\]
for every $x$ such that $F(x)<1$. Then, conditionally on the event $\{U_n=x\}$, the random vector
\[
\bigl(X_{n-k+2:n},\dots,X_{n:n}\bigr)
\]
has the same distribution as the order statistics of $k-1$ independent random variables with distribution function $F_x$.
\end{lemma}

\begin{proof}
Let
\[
Q(u)=\inf\{x\in\mathbb{R}:F(x)\ge u\}, \qquad 0<u<1,
\]
denote the generalised inverse of $F$. Since $F$ is continuous, $Q(Z_i)$ has distribution function $F$, and we may realise the sample on the same probability space by taking $X_i=Q(Z_i)$ for all $i\ge1$.

Now fix $x$ with $F(x)<1$, and write $u=F(x)$. Since $Q$ is non-decreasing, it preserves order, and therefore
\[
X_{j:n}=Q(Z_{j:n}), \qquad 1\le j\le n,
\]
almost surely. By Lemma~\ref{lem:conditional-upper-uniform-order-stats}, conditionally on $V_n=u$, the vector
\[
\bigl(Z_{n-k+2:n},\dots,Z_{n:n}\bigr)
\]
has the same distribution as the order statistics of $k-1$ independent ${\rm Uniform}(u,1)$ random variables. Applying the monotone map $Q$ coordinatewise, we conclude that conditionally on $U_n=x$, the vector
\[
\bigl(X_{n-k+2:n},\dots,X_{n:n}\bigr)
\]
has the same distribution as the order statistics of $k-1$ independent random variables with distribution function
\[
y\mapsto \mathbb{P}(Q(W)\le y),
\qquad W\sim {\rm Uniform}(u,1).
\]
For $y\ge x$,
\[
\mathbb{P}(Q(W)\le y)
=
\mathbb{P}(W\le F(y)\mid W>u)
=
\frac{F(y)-u}{1-u}
=
\frac{F(y)-F(x)}{1-F(x)}.
\]
This proves the claim.
\end{proof}

\medskip

We can now derive the transition kernel of $(U_n)$.

\begin{theorem}
\label{thm:transition-kernel-Un}
The process $(U_n)_{n\ge k}$ is a time-homogeneous Markov chain 
with respect to its natural filtration
$\mathcal{G}_n=\sigma(U_k,\dots,U_n)$, $n\ge k$.
Its transition kernel $Q_k$ is given by
\begin{equation}
\label{eq:transition-kernel-Un}
Q_k(x,A)
=
F(x)\,\delta_x(A)
+
\int_{A\cap(x,\infty)}
k\frac{(1-F(y))^{k-1}}{(1-F(x))^{k-1}}\,dF(y),
\qquad A\in\mathcal{B}(\mathbb{R}),
\end{equation}
for every $x$ such that $F(x)<1$, where $\delta_x$ denotes the Dirac measure at $x$.
\end{theorem}

\begin{proof}
Fix $n\ge k$ and let $x$ be such that $F(x)<1$. On the event $\{U_n=x\}$ there are exactly $k-1$ sample points strictly larger than $x$, namely
\[
X_{n-k+2:n},\dots,X_{n:n},
\]
and all remaining sample points are less than or equal to $x$.

Let $X_{n+1}$ be the new observation. Then:
\begin{itemize}
\item if $X_{n+1}\le x$, the running $k$th largest value remains unchanged, so $U_{n+1}=x$;
\item if $X_{n+1}>x$, then $U_{n+1}$ is the minimum of the $k$ values
\[
X_{n-k+2:n},\dots,X_{n:n},X_{n+1}.
\]
\end{itemize}

It follows that, for every $y>x$,
\begin{align*}
\mathbb{P}(U_{n+1}>y\mid U_n=x)
&=
\mathbb{P}\bigl(X_{n+1}>y,\ X_{n-k+2:n}>y \mid U_n=x\bigr).
\end{align*}
By independence of $X_{n+1}$ and $\mathcal{F}_n$,
\[
\mathbb{P}(X_{n+1}>y\mid U_n=x)=1-F(y).
\]
Moreover, by Lemma~\ref{lem:conditional-upper-order-stats}, conditionally on $U_n=x$, the variable $X_{n-k+2:n}$ is the minimum of $k-1$ independent random variables with distribution function $F_x$. Therefore
\[
\mathbb{P}(X_{n-k+2:n}>y\mid U_n=x)
=
\left(1-F_x(y)\right)^{k-1}
=
\left(\frac{1-F(y)}{1-F(x)}\right)^{k-1}.
\]
Hence, for $y>x$,
\begin{equation}
\label{eq:tail-kernel-Un}
\mathbb{P}(U_{n+1}>y\mid U_n=x)
=
(1-F(y))
\left(\frac{1-F(y)}{1-F(x)}\right)^{k-1}
=
\frac{(1-F(y))^k}{(1-F(x))^{k-1}}.
\end{equation}

Taking $y\downarrow x$ in \eqref{eq:tail-kernel-Un}, and using continuity of $F$, we obtain
\[
\mathbb{P}(U_{n+1}>x\mid U_n=x)=1-F(x),
\]
so that
\[
\mathbb{P}(U_{n+1}=x\mid U_n=x)=F(x).
\]
Thus the kernel has an atom of mass $F(x)$ at the point $x$.

Finally, for $y>x$,
\[
\mathbb{P}(x<U_{n+1}\le y\mid U_n=x)
=
\frac{(1-F(x))^k-(1-F(y))^k}{(1-F(x))^{k-1}}
=
\int_{(x,y]}
k\frac{(1-F(t))^{k-1}}{(1-F(x))^{k-1}}\,dF(t).
\]
Together with the atom at $x$, this proves~\eqref{eq:transition-kernel-Un}.
Since the right-hand side does not depend on~$n$, the kernel $Q_k$ is
time-homogeneous. It remains to verify the Markov property with respect
to the natural filtration $(\mathcal{G}_n)$.
Note that the conditional distribution
$\mathbb{P}(U_{n+1}\in\cdot\mid U_n=x)=Q_k(x,\cdot)$
was computed by conditioning only on $U_n=x$: the independence of
$X_{n+1}$ and $\mathcal{F}_n$ entered through $\mathbb{P}(X_{n+1}>y)=1-F(y)$,
while the conditional law of
$(X_{n-k+2:n},\dots,X_{n:n})$ given $U_n=x$ was supplied by
Lemma~\ref{lem:conditional-upper-order-stats}, which conditions
on $U_n$ alone and involves no other past information.
In particular, the regular conditional distribution of
$U_{n+1}$ given $U_n$ does not depend on $(U_k,\dots,U_{n-1})$, and
$(U_n)_{n\ge k}$ is a Markov chain with transition kernel~$Q_k$.
\end{proof}

\begin{remark}
\label{rem:kernel-special-case-k1Markov}
The process $(U_n)_{n \ge k}$ is in general \emph{not} Markov with respect to the larger filtration $(\mathcal F_n)_{n \ge 1}$, where $\mathcal F_n = \sigma(X_1,\ldots,X_n)$. Indeed, conditionally on $\mathcal F_n$, the full vector $Y_n=(y_1,\ldots,y_k)$ is known, and the distribution of $U_{n+1}$ depends not only on $y_1 = U_n$ but also on the remaining components $(y_2,\ldots,y_k)$. 

The Markov property with respect to $(\mathcal G_n)$ arises after averaging over the conditional distribution of $(y_2,\ldots,y_k)$ given $U_n$, which is described in Lemma~3.4. This averaging yields the transition kernel $Q_k$.
\end{remark}

%

\begin{remark}
\label{rem:kernel-special-case-k1}
For $k=1$, the process $(U_n)$ reduces to the ordinary upper record process, that is, the sequence of running maxima. In this case \eqref{eq:transition-kernel-Un} becomes
\[
Q_1(x,A)=F(x)\,\delta_x(A)+F\bigl(A\cap(x,\infty)\bigr),
\]
which is the familiar transition kernel of the running maximum process.
\end{remark}

\section{Record times associated with the process $U_n$}\label{sec:record-times}
\subsection{Type~2 $k$-record times}
\label{subsec:type2-k-record-times}

We first recall the classical definition of Type~2 $k$-record times; see, for instance, \cite{ArnoldBalakrishnanNagaraja1998,AhmadiDoostparast2008}.

\begin{definition}
\label{def:type2-k-record-times-section4}
Set $v_1^{(k)}=k$, and for $n\ge1$ define recursively
\[
v_{n+1}^{(k)}
=
\inf\{j>v_n^{(k)}:X_j>U_{v_n^{(k)}}\}.
\]
The corresponding record values are
\[
R_n^{(k)}=U_{v_n^{(k)}}, \qquad n\ge1.
\]
\end{definition}

\begin{remark}
\label{rem:interpretation-k-records}
At time $v_n^{(k)}$, the quantity $U_{v_n^{(k)}}$ represents the $k$th largest observation among $X_1,\dots,X_{v_n^{(k)}}$. The next record time occurs when a new observation exceeds this level.
\end{remark}

\subsection{Record times of the process $U_n$}
\label{subsec:record-times-Un}

We now consider the ordinary record-time sequence associated with the process $(U_n)_{n\ge k}$.

\begin{definition}
\label{def:record-times-Un-section4}
Set $\mu_1^{(k)}=k$, and for $n\ge1$ define
\[
\mu_{n+1}^{(k)}
=
\inf\{j>\mu_n^{(k)}:U_j>U_{\mu_n^{(k)}}\}.
\]
\end{definition}

\begin{remark}
\label{rem:interpretation-Un-records}
The sequence $(\mu_n^{(k)})$ is the sequence of ordinary record times of the process $(U_n)$. In particular, $(U_{\mu_n^{(k)}})$ is a strictly increasing sequence whenever $F$ is continuous.
\end{remark}
\subsection{Comparison and coincidence of the two constructions}
\label{subsec:comparison-times}

We now compare the sequences $(v_n^{(k)})$ and $(\mu_n^{(k)})$.

\begin{lemma}
\label{lem:vn-le-mun}
For all $n\ge1$,
\[
v_n^{(k)} \le \mu_n^{(k)} \quad \text{almost surely}.
\]
\end{lemma}

\begin{proof}
The proof proceeds by induction. The result is trivial for $n=1$.

Assume that $v_n^{(k)} \le \mu_n^{(k)}$. Let $j>\mu_n^{(k)}$ be such that
\[
U_j>U_{\mu_n^{(k)}}.
\]
Since $(U_m)$ is non-decreasing,
\[
U_{v_n^{(k)}} \le U_{\mu_n^{(k)}}.
\]
Hence
\[
X_j \ge U_j > U_{\mu_n^{(k)}} \ge U_{v_n^{(k)}},
\]
so $j$ is admissible in the definition of $v_{n+1}^{(k)}$. Therefore
\[
v_{n+1}^{(k)} \le \mu_{n+1}^{(k)}.
\]
\end{proof}

\medskip

\begin{lemma}
\label{lem:Un-constant-between-records}
Let $n\ge1$. For all $j$ such that
\[
v_n^{(k)} \le j < v_{n+1}^{(k)},
\]
one has
\[
U_j = U_{v_n^{(k)}}.
\]
\end{lemma}

\begin{proof}
By definition of $v_{n+1}^{(k)}$, we have
\[
X_m \le U_{v_n^{(k)}} \quad \text{for all } m=v_n^{(k)}+1,\dots,j.
\]
Hence the $k$ largest observations among $X_1,\dots,X_j$ are the same as among $X_1,\dots,X_{v_n^{(k)}}$, which yields the result.
\end{proof}

\medskip

\begin{lemma}
\label{lem:mun-le-vn}
Assume that $F$ is continuous. Then, for all $n\ge1$,
\[
\mu_n^{(k)} \le v_n^{(k)} \quad \text{almost surely}.
\]
\end{lemma}

\begin{proof}
Again we argue by induction. The case $n=1$ is trivial.

Assume $\mu_n^{(k)} \le v_n^{(k)}$. Let $j>v_n^{(k)}$ be such that
\[
X_j > U_{v_n^{(k)}}.
\]
Then $X_j$ enters the top $k$ values at time $j$, and therefore
\[
U_j > U_{v_n^{(k)}}.
\]
Since $(U_m)$ is non-decreasing,
\[
U_{\mu_n^{(k)}} \le U_{v_n^{(k)}},
\]
hence
\[
U_j > U_{\mu_n^{(k)}}.
\]
Thus $j$ is admissible for $\mu_{n+1}^{(k)}$, and we obtain
\[
\mu_{n+1}^{(k)} \le v_{n+1}^{(k)}.
\]
\end{proof}

\medskip

\begin{theorem}
\label{thm:vn-equals-mun}
Assume that $F$ is continuous. Then, for all $n\ge1$,
\[
v_n^{(k)} = \mu_n^{(k)} \quad \text{almost surely}.
\]
\end{theorem}

\begin{proof}
Combine Lemmas~\ref{lem:vn-le-mun} and \ref{lem:mun-le-vn}.
\end{proof}


\section{The Markov structure of the $k$-record process}
\label{sec:k-record-process}

In this section we combine the results of Sections~\ref{sec:markov-order-statistics} and \ref{sec:record-times} in order to obtain a Markovian description of the $k$-record process.

Recall that the sequence of $k$-record values is given by
\[
R_n^{(k)} = U_{v_n^{(k)}}, \qquad n\ge1.
\]
By Theorem~\ref{thm:vn-equals-mun}, and under the continuity assumption on $F$, we have
\[
R_n^{(k)} = U_{\mu_n^{(k)}} \quad \text{almost surely}.
\]
Thus $(R_n^{(k)})$ coincides with the sequence of record values of the Markov chain $(U_n)$.

\subsection{Transition kernel of the $k$-record chain}
\label{subsec:transition-kernel}

We first determine the transition kernel of the sequence $(R_n^{(k)})$.

The key step is the following general lemma concerning first hitting times above a level for a Markov chain with an atom.

\begin{lemma}
\label{lem:hitting-time-kernel}
Let $(X_n)_{n\ge0}$ be a time-homogeneous Markov chain on $\mathbb{R}$ with transition kernel $Q$. Fix $x\in\mathbb{R}$ and define
\[
\tau_x = \inf\{n\ge1 : X_n > x\}.
\]
Assume that
\[
Q(x,\{x\}) = p \in [0,1), \qquad Q(x,(x,\infty)) = 1-p,
\]
and that, conditionally on $X_0=x$, the event $\{X_1>x\}$ has positive probability.

Then, for every Borel set $A\subseteq (x,\infty)$,
\[
\mathbb{P}_x(X_{\tau_x}\in A)
=
\frac{Q(x,A)}{Q(x,(x,\infty))}.
\]
\end{lemma}

\begin{proof}
For $n\ge1$, we have
\[
\{ \tau_x = n \}
=
\{X_1 = x, \dots, X_{n-1}=x, X_n > x\}.
\]
By the Markov property,
\[
\mathbb{P}_x(\tau_x = n, X_n \in A)
=
p^{\,n-1} Q(x,A).
\]
Summing over $n\ge1$, we obtain
\[
\mathbb{P}_x(X_{\tau_x}\in A)
=
\sum_{n=1}^\infty p^{\,n-1} Q(x,A)
=
\frac{1}{1-p} Q(x,A).
\]
Since $1-p = Q(x,(x,\infty))$, the result follows.
\end{proof}

\medskip

We now apply this lemma to the process $(U_n)$.

\begin{theorem}
\label{thm:k-record-kernel}
Assume that $F$ is continuous. Then the process $(R_n^{(k)})_{n\ge1}$ is a time-homogeneous Markov chain. Its transition kernel $K_k$ is given by
\begin{equation}
\label{eq:k-record-kernel}
K_k(x,A)
=
\int_{A\cap(x,\infty)}
\frac{k\,(1-F(y))^{k-1}}{(1-F(x))^k}\, dF(y),
\qquad A\in\mathcal{B}(\mathbb{R}).
\end{equation}
\end{theorem}
\begin{proof}
Since $(U_n)$ is Markov only with respect to its own filtration $(\mathcal G_n)$  (see Remark \ref{rem:kernel-special-case-k1Markov}), and not with respect to $(\mathcal F_n)$. Therefore, the argument proceeds via the vector process $(Y_n)$, which is Markov with respect to $(\mathcal F_n)$.
Fix $n\ge1$ and condition on $R_n^{(k)}=x$. Using the identity
\[
R_n^{(k)} = U_{\mu_n^{(k)}},
\]
we have
\[
R_{n+1}^{(k)} = U_{\mu_{n+1}^{(k)}},
\]
where $\mu_{n+1}^{(k)}$ is the first time after $\mu_n^{(k)}$ at which $U_j>x$.

By the strong Markov property of $(Y_n)$
(Proposition~\ref{prop:vector-process-markovy}) at the $(\mathcal F_n)$-stopping time $\mu_n^{(k)}$, the post-$\mu_n^{(k)}$ process
$(Y_{\mu_n^{(k)}+m})_{m\ge0}$ depends on the past only through
$Y_{\mu_n^{(k)}}$.  By Lemma~\ref{lem:conditional-upper-order-stats},
conditionally on the first component $U_{\mu_n^{(k)}}=x$, the remaining
components of $Y_{\mu_n^{(k)}}$ are distributed as the order statistics
of $k-1$ independent random variables with distribution function~$F_x$,
independently of the past.  Hence the post-$\mu_n^{(k)}$ law of
$(U_m)_{m\ge\mu_n^{(k)}}$ coincides with that of the Markov chain
$(U_n)$ started from~$x$.  In particular,
\[
\mathbb{P}(R_{n+1}^{(k)}\in A \mid R_n^{(k)}=x)
=
\mathbb{P}_x(U_{\tau_x}\in A),
\]
where $\tau_x=\inf\{j\ge1:U_j>x\}$.

We now apply Lemma~\ref{lem:hitting-time-kernel} to the chain $(U_n)$. From Theorem~\ref{thm:transition-kernel-Un}, we have
\[
Q_k(x,\{x\}) = F(x), \qquad Q_k(x,(x,\infty)) = 1-F(x),
\]
and for $A\subseteq(x,\infty)$,
\[
Q_k(x,A)
=
\int_A k\frac{(1-F(y))^{k-1}}{(1-F(x))^{k-1}}\, dF(y).
\]

Therefore, by Lemma~\ref{lem:hitting-time-kernel},
\[
\mathbb{P}_x(U_{\tau_x}\in A)
=
\frac{Q_k(x,A)}{1-F(x)},
\]
which yields \eqref{eq:k-record-kernel}.
\end{proof}

\subsection{Consequences for record values and record times}
\label{subsec:consequences}

We summarise some immediate consequences of the preceding results.

\begin{corollary}
\label{cor:strict-increase}
Assume that $F$ is continuous. Then the sequence $(R_n^{(k)})_{n\ge1}$ is strictly increasing almost surely.
\end{corollary}

\begin{proof}
This follows from the definition of record times of the process $(U_n)$.
\end{proof}

\medskip

\begin{corollary}
\label{cor:markov-property-records}
Assume that $F$ is continuous. Then $(R_n^{(k)})$ is a time-homogeneous Markov chain with transition kernel $K_k$ given by \eqref{eq:k-record-kernel}.
\end{corollary}

\medskip

\begin{remark}
\label{rem:connection-literature}
The kernel \eqref{eq:k-record-kernel} provides a direct probabilistic description of the $k$-record process. In particular, the classical representation of $k$-records in terms of ordinary record values from the transformed distribution
\[
F_{1:k}(x)=1-(1-F(x))^k
\]
may be recovered from this kernel; see Section~\ref{sec:classical-results}.
\end{remark}

%

\section{Recovery of classical distributional results}
\label{sec:classical-results}

In this section we show how the classical distributional properties of $k$-record values follow naturally from the Markovian construction developed in the previous sections.

\subsection{The representation through $F_{1:k}$}
\label{subsec:F1k-representation}

We begin by deriving the well-known representation of $k$-record values in terms of ordinary record values from a transformed distribution.

\begin{proposition}
\label{prop:F1k-kernel}
Let $F$ be continuous and define
\[
F_{1:k}(x)=1-(1-F(x))^k.
\]
Then the transition kernel $K_k$ in \eqref{eq:k-record-kernel} coincides with the transition kernel of the ordinary record process associated with the distribution function $F_{1:k}$.
\end{proposition}

\begin{proof}
Recall that, for ordinary record values associated with a continuous distribution function $G$, the transition kernel is given by
\[
K_G(x,A)
=
\int_{A\cap(x,\infty)}
\frac{dG(y)}{1-G(x)}.
\]

In the present setting, let $G=F_{1:k}$. Then
\[
dG(y)
=
k(1-F(y))^{k-1} dF(y).
\]
Moreover,
\[
1-G(x)
=
(1-F(x))^k.
\]

Therefore,
\[
K_G(x,A)
=
\int_{A\cap(x,\infty)}
\frac{k(1-F(y))^{k-1}}{(1-F(x))^k}\, dF(y),
\]
which coincides with $K_k(x,A)$ in \eqref{eq:k-record-kernel}.
\end{proof}

\begin{corollary}
\label{cor:F1k-representation}
Assume that $F$ is continuous. Then the sequence $(R_n^{(k)})_{n\ge1}$ has the same distribution as the sequence of ordinary record values associated with the distribution function $F_{1:k}$.
\end{corollary}

\begin{proof}
The result follows from Proposition~\ref{prop:F1k-kernel}, since the distribution of a Markov chain is uniquely determined by its transition kernel and initial distribution.
\end{proof}

\begin{remark}
\label{rem:F1k-literature}
The representation in Corollary~\ref{cor:F1k-representation} is widely used in the literature; see, for example, Arnold et al.~\cite{ArnoldBalakrishnanNagaraja1998} and subsequent works such as \cite{HofmannBalakrishnan2004,AhmadiDoostparast2008}. The derivation given here shows that this representation follows directly from the Markovian structure of the process $(U_n)$.
\end{remark}

\subsection{Joint density formulae for the first $m$ $k$-records}
\label{subsec:joint-density}

We now recover the classical joint density of the first $m$ $k$-record values.

\begin{theorem}
\label{thm:joint-density-k-records}
Assume that $F$ is continuous with density $f$. Then the joint density of $(R_1^{(k)},\dots,R_m^{(k)})$ is given by
\[
f_{R_1^{(k)},\dots,R_m^{(k)}}(r_1,\dots,r_m)
=
k^m
\prod_{i=1}^m
\frac{f(r_i)}{1-F(r_i)}
\,(1-F(r_m))^k,
\]
for $r_1<\cdots<r_m$.
\end{theorem}

\begin{proof}
By Corollary~\ref{cor:F1k-representation}, the sequence $(R_n^{(k)})$ has the same distribution as the ordinary record process associated with $F_{1:k}$. The joint density of the first $m$ record values for a continuous distribution $G$ is given by
\[
\prod_{i=1}^m \frac{g(r_i)}{1-G(r_i)} \,(1-G(r_m)),
\]
for $r_1<\cdots<r_m$, where $g$ is the density of $G$.

Applying this formula with $G=F_{1:k}$ and using
\[
g(y)=k(1-F(y))^{k-1}f(y),
\qquad
1-G(y)=(1-F(y))^k,
\]
we obtain
\[
\frac{g(y)}{1-G(y)}
=
\frac{k f(y)}{1-F(y)}.
\]
Substituting into the general formula yields the result.
\end{proof}

\begin{remark}
\label{rem:joint-density-literature}
The density in Theorem~\ref{thm:joint-density-k-records} appears in several works on $k$-records; see, for instance, \cite{HofmannBalakrishnan2004,AhmadiDoostparast2008}. In those treatments it is typically introduced via the representation through $F_{1:k}$. The present approach shows that it follows directly from the Markovian construction of the $k$-record process.
\end{remark}

\medskip

\begin{remark}
\label{rem:connection-Weibull}
The representation of $k$-records through $F_{1:k}$ also underlies inferential procedures based on record values, such as those considered by Wang and Ye~\cite{WangYe2015}. In that context, transformations to exponential models and properties of record increments are often invoked. The results of the present paper provide a direct probabilistic foundation for such constructions.
\end{remark}

\section{Concluding remarks}
\label{sec:conclusion}

In this paper we have provided a direct probabilistic construction of the Type~2 $k$-record process based on the sequence of running order statistics $U_n = X_{n-k+1:n}$. We have shown that $(U_n)_{n\ge k}$ is a time-homogeneous Markov chain with an explicit transition kernel, and that, under the continuity assumption on $F$, the usual $k$-record times coincide almost surely with the record times of this process.

This identification yields a transparent description of the $k$-record values as the record values of a Markov chain. As a consequence, the transition kernel of the $k$-record process is obtained directly, and classical distributional results — including the representation in terms of the transformed distribution $F_{1:k}$ and the joint density of the first $m$ $k$-record values — follow naturally from the underlying stochastic structure.

The approach adopted here clarifies the probabilistic mechanism behind several constructions that are commonly used in the literature on $k$-records and related inferential procedures. In particular, it provides a self-contained derivation of results that are often introduced via transformation arguments or by analogy with ordinary record processes, and makes explicit, in a form adapted to the present setting, certain conditional properties of order statistics that are typically invoked without proof in the record literature.

Several directions for further investigation may be considered. It would be of interest to extend the present analysis to settings where the continuity assumption on $F$ is relaxed, as well as to other types of generalised records. Another natural direction is the study of statistical procedures based on $k$-records within the Markovian framework developed here, with particular emphasis on the role of the transition kernel in inference problems.

\medskip

\noindent
\textbf{Acknowledgements.}
This work is the fruit of several discussions with J.Hoffmann-J{\o}gensen.


\end{document}